\documentclass[11pt]{article}
\usepackage{amssymb}
\usepackage{amsmath}
\usepackage{latexsym}
\usepackage{mathrsfs}
\usepackage{algorithm}
\usepackage{algorithmic}
\usepackage{verbatim}

\newtheorem{thm}{Theorem}
\newtheorem{coro}{Corollary}

\newtheorem{prop}{Proposition}

\newtheorem{ass}{Assumption}
\newtheorem{fact}{Fact}

\newenvironment{proof}{{\bf Proof\,\,}}{\endproof\par}

\newcounter{spb}
\newcommand{\subpb}{(\alph{spb}) \addtocounter{spb}{1}}

\newcommand{\resetspb}{\setcounter{spb}{1}}

\def \openbox{$\sqcup\llap{$\sqcap$}$}
\def \endproof{\enskip \null \nobreak \hfill \openbox \par}

\newcommand{\limto}{\rightarrow}

\newcommand{\bz}{\mathbf 0}

\newcommand{\R}{\mathbb R}

\newcommand{\E}[1]{{\mathbb E}\left[ #1 \right]}

\begin{document}
\title{Non--Asymptotic Convergence Analysis of Inexact Gradient Methods for Machine Learning Without Strong Convexity}
\author{Anthony Man--Cho So\thanks{Department of Systems Engineering and Engineering Management, The Chinese University of Hong Kong, Shatin, N.~T., Hong Kong.  E--mail: {\tt manchoso@se.cuhk.edu.hk}}}
\date{\today}
\maketitle
\thispagestyle{empty}

\begin{abstract}
Many recent applications in machine learning and data fitting call for the algorithmic solution of structured smooth convex optimization problems.  Although the gradient descent method is a natural choice for this task, it requires exact gradient computations and hence can be inefficient when the problem size is large or the gradient is difficult to evaluate.  Therefore, there has been much interest in inexact gradient methods (IGMs), in which an efficiently computable approximate gradient is used to perform the update in each iteration.  Currently, non--asymptotic linear convergence results for IGMs are typically established under the assumption that the objective function is strongly convex, which is not satisfied in many applications of interest; while linear convergence results that do not require the strong convexity assumption are usually asymptotic in nature.  In this paper, we combine the best of these two types of results and establish---under the standard assumption that the gradient approximation errors decrease linearly to zero---the non--asymptotic linear convergence of IGMs when applied to a class of structured convex optimization problems.  Such a class covers settings where the objective function is not necessarily strongly convex and includes the least squares and logistic regression problems.  We believe that our techniques will find further applications in the non--asymptotic convergence analysis of other first--order methods.

\end{abstract}


\section{Introduction}
Motivated by various applications in machine learning and data fitting, there has been much interest in the design and analysis of fast algorithms for solving large--scale structured convex optimization problems recently.  A case in point is the problem of empirical risk minimization, in which one is given a set of input--output samples of a system, and the goal is to minimize the discrepancy between the observed output and the output predicted by certain parametrized model of the system.  Such a problem can be formulated as
\begin{equation}\label{eq:finite-sum}
\min _{x\in \R^{n}} \,\,\, f(x)=\frac{1}{M}\sum_{i=1}^{M} f_i(x),
\end{equation}
where $x\in\R^n$ is the parameter vector, $f_i:\R^n\limto\R$ is a convex function that measures the error or loss between the observed and predicted output of the $i$--th sample, and $M$ is the total number of available samples.  When every $f_i$ is smooth, a simple and natural approach for solving Problem~(\ref{eq:finite-sum}) is to use gradient descent.  However, this requires the computation of the full gradient $\nabla f = (1/M)\sum_{i=1}^M \nabla f_i$ in every iteration and hence can be expensive when $M$ is large or some of the $\nabla f_i$'s are difficult to evaluate.  Nevertheless, it is possible to circumvent such difficulty by exploiting the sum structure of $\nabla f$.  One strategy is to use a subset of the summands that make up the full gradient $\nabla f$ to update the solution in each iteration.  This leads to the class of incremental gradient methods, whose update formulae take the form
\begin{equation} \label{eq:IncGM-update}
x^{k+1} = x^k - \frac{\alpha_k}{|I_k|} \sum_{i \in I_k} \nabla f_i(x^k).
\end{equation}
Here, $\alpha_k > 0$ is the step size in the $k$--th iteration, and $I_k$ is a (possibly random) subset of $\mathscr{M}=\{1,2,\ldots,M\}$ chosen according to some pre--specified rules (see~\cite{B12} and the references therein for some common choices of $\{I_k\}_{k\ge0}$).  Note that the $k$--th iteration only requires the $|I_k|$ gradient values $\{\nabla f_i(x^k)\}_{i \in I_k}$.  Hence, an iteration of an incremental gradient method will generally be more efficient than that of gradient descent.  However, in order to guarantee convergence, incremental gradient methods of the form~(\ref{eq:IncGM-update}) typically require diminishing step sizes, which results in the slow (sublinear) convergence of these methods~\cite{B12}.  On the other hand, gradient descent with a constant step size can achieve fast (linear) convergence in various settings (see, e.g.,~\cite{N04}).  Thus, a problem of fundamental interest is to design methods that can enjoy both the low per--iteration complexity of incremental gradient methods and the fast convergence of gradient descent.

To approach the above problem, it is useful to consider incremental gradient methods of the form~(\ref{eq:IncGM-update}) under the framework of inexact gradient methods (IGMs).  These methods aim at minimizing an arbitrary smooth function $f$ by computing iterates $\{x^k\}_{k\ge0}$ according to the formula
\begin{equation}\label{eq:update}
x^{k+1} = x^k - \alpha_k \left( \nabla f(x^k)+e^{k+1} \right),
\end{equation}
where $G_k = \nabla f(x^k) + e^{k+1} \in \R^n$ is an approximation of the gradient $\nabla f$ at $x^k$, and $e^{k+1} = G_k - \nabla f(x^k) \in \R^n$ is the (possibly random) approximation error.  It is easy to see that the update formula~(\ref{eq:IncGM-update}) is a special case of~(\ref{eq:update}), with
$$ e^{k+1} =  \frac{1}{|I_k|} \sum_{i \in I_k} \nabla f_i(x^k) - \nabla f(x^k). $$
In fact, many other methods also fall under the IGM framework.  For details, we refer the reader to~\cite{LT93,B99,SLRB11,B12} and the discussions therein.

The rationale behind the update~(\ref{eq:update}) is that an approximate gradient can often be computed very efficiently.  Thus, IGMs could have significant computational gain in each iteration.  However, the convergence rates of such methods depend crucially on the choice of step sizes $\{\alpha_k\}_{k\ge0}$ and the magnitude of the error vectors $\{e^k\}_{k\ge1}$.  Many recent works on the convergence analysis of IGMs have focused on the case where the step sizes are constant and developed conditions under which a {\it non--asymptotic} linear rate of convergence can be achieved.  For instance, Blatt et al.~\cite{BHG07} proposed an IGM, called the incremental aggregated gradient method, for solving Problem~(\ref{eq:finite-sum}) and showed that it converges linearly when $f$ is a strongly convex quadratic function.  Later, Le Roux et al.~\cite{LRSB12} developed another IGM, called the stochastic average gradient method, for solving Problem~(\ref{eq:finite-sum}) and proved that it converges linearly when $f$ is strongly convex.  It is interesting to note that neither of the above results require diminishing step sizes or diminishing error norms.  On another front, Byrd et al.~\cite{BCNW12} established the linear convergence of a certain instantiation of the incremental gradient method~(\ref{eq:IncGM-update}) when $f$ is strongly convex and has a bounded Hessian.  For the general IGM~(\ref{eq:update}) with constant step sizes, Friedlander and Schmidt~\cite{FS12} (see also~\cite{SLRB11}) showed that it converges (sub)linearly if $f$ is strongly convex and the squared error norms $\{\|e^k\|_2^2\}_{k\ge1}$ decrease (sub)linearly to zero.  It should be noted that all the aforementioned linear convergence results apply only to problems with a strongly convex objective.  As such, they do not cover several important applications such as least squares and logistic regression.  Although many works have studied the non--asymptotic convergence rates of IGMs when the objective function is not strongly convex, the best known rate is only sublinear (see, e.g.,~\cite{BB07,NJLS09,BM11,ABRW12,RSS12}).

In another direction, there have been some early works that establish the {\it asymptotic} linear convergence of IGMs without requiring the objective function to be strongly convex.  For instance, Luo and Tseng~\cite{LT92,LT93} showed that if the error norms satisfy $\|e^{k+1}\|_2 = O(\|x^k - x^{k+1}\|_2)$, then the IGM~(\ref{eq:update}) with step sizes bounded away from zero has an asymptotic linear rate of convergence when applied to certain structured convex optimization problems.  In particular, the result of Luo and Tseng applies to least squares and logistic regression.  However, it should be noted that the condition on the error norms as stated above is rather strong, for it implies that the objective values of the iterates are strictly decreasing.  Subsequently, Li~\cite{L93} showed that the asymptotic linear convergence result of Luo and Tseng still holds under the weaker condition that the error norms decrease linearly to zero.  This shows that even with large gradient approximation errors in the early iterations---which typically yields computational savings but may lead to an increase in the objective value in some iterations---the IGM can still converge quickly.

In view of the above discussion, our main contribution in this paper is to show that the IGM~(\ref{eq:update}) with step sizes bounded away from zero has a non--asymptotic (sub)linear rate of convergence when applied to a class of structured convex optimization problems (which includes least squares and logistic regression), provided that the squared error norms $\{\|e^k\|_2^2\}_{k\ge1}$ decrease (sub)linearly to zero.  In particular, our linear convergence result extends those in~\cite{LT92,L93,LT93} in that it holds {\it non--asymptotically}, and those in~\cite{BHG07,SLRB11,BCNW12,FS12} in that it covers cases where the objective function is {\it not necessarily strongly convex}.  A key step in our analysis is to develop a global version of the error bound in~\cite{LT92}.  Such a global error bound provides a way to measure the progress of the IGM~(\ref{eq:update}) in {\it every} iteration and not just those iterations that are close to convergence.  This, together with the powerful convergence analysis framework developed by Luo and Tseng~\cite{LT93}, allows us to establish the desired non--asymptotic convergence rate results.

Finally, we remark that as we are preparing this manuscript, we become aware of the work~\cite{WL13}, in which the authors develop global error bounds similar to ours to study the non--asymptotic convergence rates of feasible descent methods.  However, our work differs from theirs in two important aspects.  First, the analysis in~\cite{WL13} requires the objective values of the iterates to be strictly decreasing.  As mentioned earlier, this can be quite restrictive.  On the other hand, our analysis does not have such a requirement.  Secondly, in the context of risk minimization, the analysis in~\cite{WL13} applies only to the class of globally strongly convex loss functions, which precludes many commonly used loss functions such as the logistic loss $u \mapsto \log(1+\exp(-u))$.  By contrast, our analysis only requires the loss function to be strongly convex on compact subsets, and hence it applies to a much wider class of loss functions.

\section{Preliminaries}
\subsection{Basic Setup and Observations}
In this paper, we focus on the following unconstrained convex optimization problem:
\begin{equation}\label{basic-problem}
 \min_{x\in \mathbb{R}^n} \,\,\, f(x) = g(Ex),
\end{equation}
where $E\in\mathbb{R}^{m\times n}$ is a linear operator and $g: \R^m\limto\R$ is a function satisfying the following assumptions:
\begin{ass}\label{ass:general}{\quad}
\begin{enumerate}

\item[\subpb] The function $g$ is continuously differentiable on $\R^m$ and its gradient $\nabla g$ is Lipschitz continuous with parameter $L>0$ on $\R^m$, i.e.,
$$ \| \nabla g(u) - \nabla g(v) \|_2 \le L\|u-v\|_2 \quad\mbox{for } u,v \in \R^m. $$


\item[\subpb] The function $g$ is strictly convex on $\R^m$.


\end{enumerate}
\resetspb
\end{ass}
The above setup is motivated by the empirical risk minimization problem~(\ref{eq:finite-sum}).  Indeed, in many applications, the prediction error of the $i$--sample $f_i$ can be expressed as $f_i(x) = \ell(b_i,a_i^Tx)$, where $(a_i,b_i) \in \R^n \times \R$ is the $i$--th input--output sample, and $\ell:\R\times\R \limto \R$ is a loss function.  Thus, the objective function $f$ in Problem~(\ref{eq:finite-sum}) can be put into the form $f(x)=g(Ex)$, where $E$ is an $M \times n$ matrix whose $i$--th row is $a_i^T$, and $g:\R^M \limto \R$ is given by
$$ g(y) = \frac{1}{M} \sum_{i=1}^M \ell(b_i,y_i). $$
For instance, by taking $\ell$ to be the square loss $\ell(u,v) = (u-v)^2$, we obtain the least squares regression problem
\begin{equation} \label{eq:lsq}
\min_{x \in \R^n} \,\,\, f(x) = \frac{1}{M} \sum_{i=1}^M \underbrace{ \left( a_i^Tx - b_i \right)^2 }_{\ell(b_i,a_i^Tx)}.
\end{equation}
On the other hand, by using the logistic loss $\ell(u,v) = \log( 1 + \exp(-uv) )$, we arrive at the logistic regression problem
\begin{equation} \label{eq:logistic}
\min_{x \in \R^n} \,\,\, f(x) = \frac{1}{M} \sum_{i=1}^M \underbrace{ \log \left( 1 + \exp(-b_ia_i^Tx) \right) }_{\ell(b_i,a_i^Tx)}.
\end{equation}
In both examples, it is easy to verify that the corresponding $g$ satisfies Assumption~\ref{ass:general}.


Going back to Problem~(\ref{basic-problem}), we note that the strict convexity of $g$ on $\R^m$ does not necessarily imply the strict convexity of $f$ on $\R^n$, as $E$ may not have full column rank.  Now, by Assumption~\ref{ass:general}(a), it is easy to verify that $\nabla f$ is Lipschitz continuous with parameter $L_f = L \|E\|^2$ on $\R^n$, where $\|E\|=\sup_{\|x\|_2=1}\|Ex\|_2$ is the spectral norm of $E$.  Indeed, for $x,y \in \R^n$, we have $\nabla f(x)=E^T\nabla g(Ex)$, and
\begin{eqnarray*}
\| \nabla f(x)-\nabla f(y) \|_2 &=& \left\| E^T \left( \nabla g(Ex) - \nabla g(Ey) \right) \right\|_2 \\
&\le& L\cdot\|E\|\cdot\|Ex-Ey\|_2 \\
&\le& L\|E\|^2\|x-y\|_2.
\end{eqnarray*}
Finally, let $\mathcal{X}$ denote the set of optimal solutions to Problem (\ref{basic-problem}).  We make the following assumption concerning $\mathcal{X}$:
\begin{ass} \label{ass:non-empty}
The optimal solution set $\mathcal{X}$ is non--empty.
\end{ass}

Assumption \ref{ass:non-empty} implies that the optimal value $f_{\rm min}$ of Problem (\ref{basic-problem}) is finite and bounded from below.  This, together with Assumption \ref{ass:general}(b), yields the following simple but useful result:
\begin{prop} \label{prop:inv}
The map $x \mapsto Ex$ is invariant over the optimal solution set $\mathcal{X}$; i.e., there exists a $t^* \in \R^m$ such that $Ex^*=t^*$ for all $x^* \in \mathcal{X}$.
\end{prop}
\begin{proof}
Let $x^*,y^* \in \mathcal{X}$ be arbitrary.  By the convexity of $\mathcal{X}$, we have $(x^*+y^*)/2 \in \mathcal{X}$.  Hence, the convexity of $f$ and optimality of $x^*,y^*$ imply that $f((x^*+y^*)/2) = (f(x^*)+f(y^*))/2$, or equivalently, $g((Ex^*+Ey^*)/2) = (g(Ex^*)+g(Ey^*))/2$.  Since $g$ is strictly convex on $\R^m$, we conclude that $Ex^*=Ey^*$, as desired.
\end{proof}

\subsection{Inexact Gradient Methods}
One approach for solving Problem~(\ref{basic-problem}) is to use inexact gradient methods (IGMs), which compute iterates according to the formula (\ref{eq:update}).  Our goal is to establish the convergence rate of the IGM~(\ref{eq:update}) under Assumptions \ref{ass:general} and \ref{ass:non-empty} and various conditions on the error sequence $\{e^k\}_{k\ge1}$.  We allow for the possibility that $e^1,e^2,\ldots$ are random, in which case the iterates $x^1,x^2,\ldots$ will also be random.  To simplify the exposition, we assume that the step sizes $\{\alpha_k\}_{k\ge0}$ in (\ref{eq:update}) are constant and equal to $1/L_f$.  However, it should be noted that our analysis can also be applied to the case where the step sizes $\{a_k\}_{k\ge0}$ satisfy $\lim\inf_{k\ge0} \alpha_k >0$.

The first step of our convergence analysis is to understand the behavior of the (possibly random) sequence $\{f(x^k)\}_{k\ge0}$.  It is well known that when there is no error (i.e.,~$e^k=\bz$ for all $k\ge0$), the IGM (\ref{eq:update}) will generate a sequence of iterates $\{x^k\}_{k\ge0}$ whose associated objective values $\{f(x^k)\}_{k\ge0}$ are monotonically decreasing~\cite{LP66}.  However, this may not be true in the presence of errors.  The following proposition provides a bound on the difference of the objective values of two successive iterates in terms of the error size.  Its proof is standard and can be found in Appendix~\ref{app:pf-suff-decrease}.
\begin{prop}\label{prop:suff-decrease}
The sequence $\{x^k\}_{k\ge0}$ generated by the IGM (\ref{eq:update}) satisfies
\begin{equation} \label{eq:lip}
f(x^k) - f(x^{k+1}) \ge \frac{L_f}{2} \|x^k - x^{k+1} \|_2^2 - \|e^{k+1}\|_2 \|x^k-x^{k+1}\|_2 
 \end{equation}
for all $k\ge0$.
\end{prop}
An immediate consequence of Proposition \ref{prop:suff-decrease} is the following:
\begin{coro} \label{cor:iter-bd}
(cf.~\cite{L93}) For all $k\ge0$, we have
\begin{enumerate}
\item[\subpb]
$$ \|x^k - x^{k+1}\|_2^2 \le \frac{4}{L_f}\left( f(x^k) - f(x^{k+1}) + \frac{1}{L_f} \|e^{k+1}\|_2^2 \right), $$
\item[\subpb]
$$ 0 \le f(x^{k+1}) - f_{\rm min} \le f(x^k) - f_{\rm min} + \frac{1}{2L_f}\|e^{k+1}\|_2^2. $$
\end{enumerate}
\resetspb
\end{coro}
Although the error sequence $\{e^k\}_{k\ge1}$ can be random, it should be noted that the inequalities in both Proposition \ref{prop:suff-decrease} and Corollary \ref{cor:iter-bd} hold for every realization of $\{e^k\}_{k\ge1}$.  

\section{Error Bound Condition}
Since we are interested in analyzing the convergence rate of the IGM (\ref{eq:update}), we need a measure to quantify its progress towards optimality.  One natural candidate would be the distance to the optimal solution set $\mathcal{X}$.  Indeed, since $\mathcal{X}$ is non--empty, convex, and closed (the closedness of $\mathcal{X}$ follows from the continuity of $g$), every $x\in\R^n$ has a unique projection $\bar{x} \in \mathcal{X}$ onto $\mathcal{X}$, and hence the measure $x \mapsto \mbox{dist}(x,\mathcal{X})$, where
$$ \mbox{dist}(x,\mathcal{X}) = \min_{y \in \mathcal{X}} \|x-y\|_2, $$
is well defined.  Despite its intuitive appeal, the measure $\mbox{dist}(\cdot,\mathcal{X})$ has one major disadvantage, namely, it is not easy to compute.  An alternative would be to consider the norm of the gradient $x\mapsto\|\nabla f(x)\|_2$, which is motivated by the fact that the optimality condition of (\ref{basic-problem}) is $\nabla f(x)=\bz$.  However, since $\|\nabla f(\cdot)\|_2$ is only a surrogate of $\mbox{dist}(\cdot,\mathcal{X})$, we need to establish a relationship between them.  Towards that end, consider the set
$$ \mathcal{S}_B = \left\{ y \in \R^m: \|y-t^*\|_2 \le B \right\}, $$
where $B>0$ is arbitrary.  We then have the following theorem:
\begin{thm}\label{thm:glo-err-bd}
Suppose that both Assumptions \ref{ass:general} and \ref{ass:non-empty} hold for Problem (\ref{basic-problem}).  Suppose further that $g$ is strongly convex on $\mathcal{S}_B$ for some $B>0$; i.e.,
$$ g(y) - g(z) \ge (y-z)^T\nabla g(z) + \sigma_B\|y-z\|_2^2 \quad\mbox{for all } y,z \in \mathcal{S}_B. $$
Then, there exists a $\tau_B>0$ such that 
\begin{equation}\label{eq:glo-err-bd}
\mbox{dist}(x,\mathcal{X}) \leq \tau_B \|\nabla f(x)\|_2
\end{equation}
for all $x\in\R^n$ satisfying $Ex \in \mathcal{S}_B$.
\end{thm}
Condition (\ref{eq:glo-err-bd}) is a so--called {\it error bound} for Problem (\ref{basic-problem}).  The proof of Theorem \ref{thm:glo-err-bd} relies on the following proposition, whose proof can be found in Appendix \ref{app:pre-eb}:
\begin{prop} \label{prop:pre-eb}
There exist an $\omega > 0$ such that for any $x\in\R^n$, there exists an $x^* \in \mathcal{X}$ satisfying
\begin{equation} \label{eq:pre-eb}
\|x - x^*\|_2 \le \omega \left( \|\nabla f(x)\|_2 + \|Ex - t^* \|_2 \right).
\end{equation}
\end{prop}
\noindent{\bf Proof of Theorem \ref{thm:glo-err-bd}} The argument is similar to that in~\cite{LT92}.  Let $x \in \R^n$ be such that $Ex \in \mathcal{S}_B$.  The strong convexity of $g$ on $\mathcal{S}_B$ implies that
\begin{eqnarray*}
   \sigma_B \|Ex-t^*\|_2^2 &\le& g(Ex) - g(t^*) - (Ex-t^*)^T\nabla g(t^*), \\
   \noalign{\smallskip}
   \sigma_B \|Ex-t^*\|_2^2 &\le& g(t^*) - g(Ex) - (t^*-Ex)^T\nabla g(Ex).
\end{eqnarray*}
Adding the above two inequalities yields
\begin{eqnarray}
   2\sigma_B\|Ex-t^*\|_2^2 &\le& (Ex-t^*)^T(\nabla g(Ex) - \nabla g(t^*)) \nonumber \\
   \noalign{\medskip}
   &=& (x-x^*)^T(\nabla f(x) - \nabla f(x^*)) \nonumber \\
   \noalign{\medskip}
   &=& (x-x^*)^T\nabla f(x) \nonumber \\
   \noalign{\medskip}
   &\le& \|x-x^*\|_2  \|\nabla f(x)\|_2, \label{eq:strong-cvx}
\end{eqnarray}
where the second equality follows from the fact that $\nabla f(x^*)=\bz$.  In addition, by Proposition~\ref{prop:pre-eb}, there exists an $x^* \in \mathcal{X}$ such that (\ref{eq:pre-eb}) holds.  Hence, using (\ref{eq:pre-eb}) and (\ref{eq:strong-cvx}), we compute
\begin{eqnarray}
\|x-x^*\|_2^2 &\le& \omega^2 \left( \|\nabla f(x)\|_2 + \|Ex - t^* \|_2 \right)^2 \nonumber \\
\noalign{\medskip}
&\le& 2\omega^2 \left( \|\nabla f(x)\|_2^2 + \|Ex - t^* \|_2^2 \right) \nonumber \\
\noalign{\medskip}
&\le& 2\omega^2\left[ \|\nabla f(x)\|_2 \left( \|\nabla f(x)\|_2 + \frac{1}{2\sigma_B}\|x-x^*\|_2 \right) \right] \nonumber \\
\noalign{\medskip}
&\le& 2\omega^2\left[ \|\nabla f(x)\|_2 \left( \left(1+\frac{\omega}{2\sigma_B}\right) \|\nabla f(x)\|_2 + \frac{\omega}{2\sigma_B} \|Ex-t^*\|_2 \right) \right]. \label{eq:pre-eb2}
\end{eqnarray}
Since $\|Ex-t^*\|_2 = \|Ex-Ex^*\|_2 \le \|E\|\cdot\|x-x^*\|_2$, it follows from (\ref{eq:pre-eb2}) that
$$ \|Ex-t^*\|_2^2 \le 2\|E\|^2\omega^2\left( 1+\frac{\omega}{2\sigma_B} \right) \left[ \|\nabla f(x)\|_2 ( \|\nabla f(x)\|_2 + \|Ex-t^*\|_2 ) \right]. $$
Let $\gamma_B = 2\|E\|^2\omega^2(1+\omega/2\sigma_B)$.  Then, the above inequality is of the form $U^2 \le \gamma_B\left(V(U+V)\right)$ with $U,V\ge 0$.  This implies that $U \le \bar{\gamma}_B V$, where $\bar{\gamma}_B = \left( \gamma_B + \sqrt{\gamma_B^2+4\gamma_B} \right) \big/ 2$.  Hence, we obtain $\|Ex-t^*\|_2 \le \bar{\gamma}_B \|\nabla f(x)\|_2$. This, together with Proposition \ref{prop:pre-eb}, yields (\ref{eq:glo-err-bd}) with $\tau_B = \omega(1+\bar{\gamma}_B)$. \endproof

\section{Convergence Analysis of the IGM} \label{sec:conv-anal}
Armed with Theorem \ref{thm:glo-err-bd}, we are now ready to analyze the convergence rate of the IGM (\ref{eq:update}) in the following two scenarios:
\begin{itemize}
\item[(S1)] The function $g$ is strongly convex on $\R^m$.

\item[(S2)] The function $g$ is strongly convex on $\mathcal{S}_B$ for all $B \in (0,\infty)$, and the (possibly random) error sequence $\{e^k\}_{k\ge1}$ satisfies $\sum_{k=1}^{\infty} \|e^k\|_2^2 \le \Gamma$ for some $\Gamma \in (0,\infty)$.
\end{itemize}
It is easy to verify that the least squares and logistic regression problems (see~(\ref{eq:lsq}) and (\ref{eq:logistic})) fall under scenarios (S1) and (S2), respectively.

The key step of the analysis is the following proposition, which establishes a recurrence relation between $f(x^{k+1})-f_{\rm min}$ and $f(x^k)-f_{\rm min}$.  The convergence rates of the sequences $\{f(x^k)\}_{k\ge0}$ and $\{x^k\}_{k\ge0}$ will then follow.
\begin{prop} \label{prop:cost-to-go}
Suppose that both Assumptions \ref{ass:general} and \ref{ass:non-empty} hold for Problem (\ref{basic-problem}).  Furthermore, suppose that either (S1) or (S2) holds.  Let $\{x^k\}_{k\ge0}$ be the sequence generated by the IGM (\ref{eq:update}), where $\alpha_k=1/L_f$ for all $k\ge0$, and the initial iterate $x^0$ is deterministic.  Then, there exist $\kappa, \nu>0$, which do not depend on the realization of the error sequence $\{e^k\}_{k\ge1}$, such that for all $k\ge0$, 
\begin{equation}\label{ieq:err-bd}
 \mbox{dist}(x^k,\mathcal{X}) \le \kappa \left( \|x^k - x^{k+1}\|_2 + \|e^{k+1}\|_2 \right)
\end{equation}
and
\begin{equation}\label{ieq:cost-to-go}
 f(x^{k+1}) - f_{\rm min} \le \nu\left( \|x^k - x^{k+1}\|_2 + \|e^{k+1}\|_2 \right)^2.
\end{equation}
Consequently, there exist $\mu \in (0,1)$ and $\delta>0$, which do not depend on the realization of $\{e^k\}_{k\ge1}$, such that for all $k\ge0$, 
\begin{equation}\label{ieq:decrease}
 f(x^{k+1}) - f_{\rm min} \le \mu \left( f(x^k) - f_{\rm min} \right) + \delta \|e^{k+1}\|_2^2.
\end{equation}
\end{prop}
\begin{proof}
Let us first verify that in both scenarios (S1) and (S2), there exists a $B>0$ such that $Ex^k \in \mathcal{S}_B$ for all $k\ge0$, and that $g$ is strongly convex on $\mathcal{S}_B$.  In scenario (S1), we can simply set $B=\infty$ to get $\mathcal{S}_B = \R^m$.  In scenario (S2), observe that Corollary \ref{cor:iter-bd}(b) implies
$$ 0 \le f(x^k) - f_{\rm min} \le f(x^0) - f_{\rm min} + \sum_{j=1}^k \|e^j\|_2^2 \le f(x^0) - f_{\rm min} + \Gamma $$
for all $k\ge0$.  Hence, by~\cite[Fact 4.1]{T91}, the sequence $\{Ex^k\}_{k\ge0}$ is bounded.  Consequently, there exists a $B \in (0,\infty)$, which does not depend on the realization of $\{x^k\}_{k\ge0}$, such that $Ex^k \in \mathcal{S}_B$, and $g$ is strongly convex on $\mathcal{S}_B$.

The above argument implies that Theorem \ref{thm:glo-err-bd} applies to both scenarios (S1) and (S2).  Hence, there exists a $\tau>0$, which does not depend on the realization of $\{e^k\}_{k\ge1}$, such that
\begin{eqnarray*}
\mbox{dist}(x^k,\mathcal{X}) &\le& \tau\|\nabla f(x^k)\|_2 \\
\noalign{\medskip}
&=& \tau\left\| L_f(x^k-x^{k+1})-e^{k+1} \right\|_2 \\
\noalign{\medskip}
&\le& \tau \tilde{L}_f \left( \|x^k-x^{k+1}\|_2 + \|e^{k+1}\|_2 \right)
\end{eqnarray*}
for all $k\ge0$, where $\tilde{L}_f = \max\{1,L_f\}$, and the equality follows from the update formula (\ref{eq:update}).  This establishes (\ref{ieq:err-bd}) with $\kappa = \tau\tilde{L}_f$.

To prove (\ref{ieq:cost-to-go}), let $\bar{x}^k=\arg\min_{y\in\mathcal{X}}\|y-x^k\|_2$, where $k=0,1,\ldots$.  Then, we have $f(\bar{x}^k) = f_{\rm min}$, and the above derivation implies that
$$ \|\bar{x}^k - x^k\|_2 \le \kappa \left( \|x^k-x^{k+1}\|_2 + \|e^{k+1}\|_2 \right). $$
Moreover, by the Mean Value Theorem, there exists a $\hat{x}^k \in [\bar{x}^k,x^{k+1}]$ such that
$$ f(x^{k+1}) - f(\bar{x}^k) = \nabla f(\hat{x}^k)^T(x^{k+1}-\bar{x}^k). $$
Hence, it follows that
\begin{eqnarray*}
& & f(x^{k+1}) - f_{\rm min} \\
\noalign{\medskip}
&=& \left( \nabla f(\hat{x}^k) - \nabla f(x^k) \right)^T(x^{k+1}-\bar{x}^k) + \nabla f(x^k)^T(x^{k+1}-\bar{x}^k) \\
\noalign{\medskip}
&\le& \left\| \nabla f(\hat{x}^k) - \nabla f(x^k) \right\|_2 \|x^{k+1}-\bar{x}^k\|_2  + L_f\left( x^k - x^{k+1} - \frac{1}{L_f} e^{k+1} \right)^T(x^{k+1} - \bar{x}^k) \\
\noalign{\medskip}
&\le& \left[ L_f \left( \|\hat{x}^k-x^k\|_2 + \|x^k - x^{k+1}\|_2 \right) + \|e^{k+1}\|_2 \right] \times \|x^{k+1}-\bar{x}^k\|_2 \\
\noalign{\medskip}
&\le& \left[ L_f \left( \|x^k - x^{k+1}\|_2 + \|x^k - \bar{x}^k\|_2 + \|x^k - x^{k+1}\|_2 \right) + \|e^{k+1}\|_2 \right] \\
\noalign{\medskip}
&\quad\times& \left( \|x^k-x^{k+1}\|_2 + \|x^k - \bar{x}^k\|_2 \right) \\
\noalign{\medskip}
&\le& \left[  (2+\kappa)L_f \|x^k-x^{k+1}\|_2 + \left( 1 + \kappa L_f \right) \|e^{k+1}\|_2 \right] \times (1+\kappa) \left( \|x^k-x^{k+1}\|_2 + \|e^{k+1}\|_2 \right) \\
\noalign{\medskip}
&\le& \nu \left( \|x^k-x^{k+1}\|_2 + \|e^{k+1}\|_2 \right)^2,
\end{eqnarray*}
where $\nu = (1+\kappa)\max\{(2+\kappa)L_f,1+\kappa L_f\}$.  

Finally, to prove (\ref{ieq:decrease}), observe that by (\ref{ieq:cost-to-go}) and Corollary \ref{cor:iter-bd}(a),
\begin{eqnarray*}
f(x^{k+1}) - f_{\rm min} &\le& \nu\left( \|x^k - x^{k+1}\|_2 + \|e^{k+1}\|_2 \right)^2 \\
\noalign{\medskip}
&\le& 2\nu \left[ \frac{4}{L_f} \left( f(x^k) - f(x^{k+1}) + \frac{1}{L_f}\|e^{k+1}\|_2^2 \right) + \|e^{k+1}\|_2^2 \right] \\
\noalign{\medskip}
&=& \frac{8\nu}{L_f} \left[ \left( f(x^k) - f_{\rm min} \right) - \left( f(x^{k+1}) - f_{\rm min} \right) + \frac{L_f+L_f^{-1}}{4}\|e^{k+1}\|_2^2 \right].
\end{eqnarray*}
Upon rearranging, the desired result (\ref{ieq:decrease}) follows with
\begin{equation}\label{eq:parameter}
 \mu = \frac{8\nu/L_f}{1+(8\nu/L_f)} \in (0,1), \quad \delta = \frac{2\nu \left[ 1+(1/L_f^2) \right]}{1+(8\nu/L_f)} > 0.
\end{equation}
\end{proof}

\medskip
Proposition \ref{prop:cost-to-go} immediately leads to the following corollary, whose proof can be found in Appendix \ref{app:fn-root}:
\begin{coro}\label{cor:fn-root}
Under the setting of Proposition \ref{prop:cost-to-go}, there exists a $\beta>0$, which does not depend on the realization of $\{e^k\}_{k\ge1}$, such that for all $k\ge0$, we have
$$
 f(x^k)-f_{\rm min}\leq \mu^k \left( f(x^0)-f_{\rm min} \right) + \delta\sum_{j=1}^k \mu^{k-j} \|e^j\|_2^2
$$
and
$$ \left| f(x^{k+1}) - f(x^k) \right| \le \beta \sum_{j=1}^{k+1} \left( \mu^{k+1-j} \|e^j\|_2^2 + \mu^k \right), $$
where $\mu \in (0,1)$ and $\delta > 0$ are given by (\ref{eq:parameter}).
\end{coro}
Corollary \ref{cor:fn-root} shows that the rate at which the objective values $\{f(x^k)\}_{k\ge0}$ converge to the optimal value $f_{\rm min}$ is largely determined by the rate at which the norms of the error vectors $\{e^k\}_{k\ge1}$ decrease to zero.  However, since the objective function $f$ is not necessarily strongly convex, the convergence rate of the objective values does not automatically translate into the convergence rate of the iterates.  The following result shows how the latter can be determined by utilizing Proposition \ref{prop:cost-to-go}, Corollary \ref{cor:iter-bd}(a) and Corollary \ref{cor:fn-root}:
\begin{thm} \label{thm:conv-rate-gen}
Under the setting of Proposition \ref{prop:cost-to-go}, there exist $\lambda_1,\lambda_2 > 0$, which do not depend on the realization of $\{e^k\}_{k\ge1}$, such that
\begin{equation} \label{eq:iter-conv}
\|x^k - x^{k+1}\|_2 \le \lambda_1 \left[ \sum_{j=1}^{k+1} \mu^{(k+1-j)/2} \|e^j\|_2 + \left( \frac{1+\mu}{2} \right)^{k/2} \right]
\end{equation}
and 
\begin{equation} \label{eq:dist-to-opt}
\mbox{dist}(x^k,\mathcal{X}) \le \lambda_2 \left[ \sum_{j=1}^{k+1} \mu^{(k+1-j)/2} \|e^j\|_2 + \left( \frac{1+\mu}{2} \right)^{k/2} \right]
\end{equation}
for all $k\ge0$, where $\mu \in (0,1)$ is given by (\ref{eq:parameter}).
\end{thm}
\begin{proof}
By Corollary \ref{cor:iter-bd}(a) and Corollary \ref{cor:fn-root}, we have
\begin{eqnarray*}
\|x^k - x^{k+1}\|_2^2 &\le& \frac{4}{L_f}\left( \left| f(x^k) - f(x^{k+1}) \right| + \frac{1}{L_f} \|e^{k+1}\|_2^2 \right) \\
\noalign{\medskip}
&\le& \frac{4}{L_f}\left[ \beta \sum_{j=1}^{k+1} \left( \mu^{k+1-j} \|e^j\|_2^2 + \mu^k \right) + \frac{1}{L_f}\|e^{k+1}\|_2^2 \right] \\
\noalign{\medskip}
&\le& \frac{4}{L_f} \left[ \left( \beta + \frac{1}{L_f} \right) \sum_{j=1}^{k+1} \mu^{k+1-j} \|e^j\|_2^2 + \beta(k+1)\mu^k \right].
\end{eqnarray*}
Since there exists a $\gamma>0$ such that $(k+1)\mu^k \le \gamma((1+\mu)/2)^k$ for all $k\ge0$, we see that
$$ \|x^k - x^{k+1}\|_2^2 \le \lambda_1^2 \left[ \sum_{j=1}^{k+1} \mu^{k+1-j} \|e^j\|_2^2 + \left( \frac{1+\mu}{2} \right)^k \right], $$
where $\lambda_1^2 = (4/L_f)\max\{ \beta+(1/L_f), \beta\gamma \}$.  The desired result (\ref{eq:iter-conv}) then follows from the fact that $\sqrt{a+b} \le \sqrt{a}+\sqrt{b}$ for all $a,b\ge0$.

Now, using (\ref{ieq:err-bd}) and (\ref{eq:iter-conv}), we obtain
\begin{eqnarray*}
\mbox{dist}(x^k,\mathcal{X}) &\le& \kappa \left( \|x^k - x^{k+1}\|_2 + \|e^{k+1}\|_2 \right) \\
\noalign{\medskip}
&\le& \kappa \left( \lambda_1 \left[ \sum_{j=1}^{k+1} \mu^{(k+1-j)/2} \|e^j\|_2 + \left( \frac{1+\mu}{2} \right)^{k/2} \right] + \|e^{k+1}\|_2 \right) \\
\noalign{\medskip}
&\le& \lambda_2 \left[ \sum_{j=1}^{k+1} \mu^{(k+1-j)/2} \|e^j\|_2 + \left( \frac{1+\mu}{2} \right)^{k/2} \right], 
\end{eqnarray*}
where $\lambda_2 = \kappa(1+\lambda_1)$.  This establishes (\ref{eq:dist-to-opt}), and the proof of Theorem \ref{thm:conv-rate-gen} is completed.
\end{proof}

\medskip
From Theorem \ref{thm:conv-rate-gen}, we see that the rate at which the norms of the error vectors $\{e^k\}_{k\ge1}$ decrease to zero again plays an important role---this time in determining the rate at which the iterates $\{x^k\}_{k\ge0}$ converge to an element in the optimal set $\mathcal{X}$.  As a direct application of Corollary~\ref{cor:fn-root} and Theorem~\ref{thm:conv-rate-gen}, we have the following corollary, whose proof can be found in Appendix \ref{app:iter-conv}:
\begin{coro} \label{cor:iter-conv} Consider the setting of Proposition \ref{prop:cost-to-go}.
\begin{enumerate}
\item[\subpb] (Sublinear Convergence) Suppose that for some $\rho>0$, we have $\|e^k\|_2^2 \le B_k = O\left( 1/k^{1+\rho} \right)$ for all $k\ge1$.  Then, the sequence of iterates $\{x^k\}_{k\ge0}$ satisfies
$$ f(x^{k+1}) - f_{\rm min} \le O\left( \frac{1}{(k+1)^{1+\rho}} \right) $$
and
$$ \|x^k - x^{k+1}\|_2 \le O\left( \frac{1}{(k+1)^{(1+\rho)/2}} \right), \quad \mbox{dist}(x^k,\mathcal{X}) \le O\left( \frac{1}{(k+1)^{(1+\rho)/2}} \right) $$
for all $k\ge0$.  In particular, the sequence $\{f(x^k)\}_{k\ge0}$ (resp.~$\{x^k\}_{k\ge0}$) converges at least sublinearly to $f_{\rm min}$ (resp.~an element in $\mathcal{X}$).

\item[\subpb] (Linear Convergence) Suppose that for some $\rho \in (0,1)$, we have $\|e^k\|_2^2 \le B_k = O(\rho^k)$ for all $k\ge1$.  Then, there exists a $c \in (0,1)$ such that the sequence of iterates $\{x^k\}_{k\ge0}$ satisfies
$$ f(x^{k+1}) - f_{\rm min} \le O( c^{2(k+1)} ) $$
and
$$ \|x^k - x^{k+1}\|_2 \le O(c^k), \quad \mbox{dist}(x^k,\mathcal{X}) \le O(c^k) $$
for all $k\ge0$.  In particular, the sequence $\{f(x^k)\}_{k\ge0}$ (resp.~$\{x^k\}_{k\ge0}$) converges at least linearly to $f_{\rm min}$ (resp.~an element in $\mathcal{X}$).
\end{enumerate}
\resetspb
\end{coro}
In the context of inexact gradient methods, Corollary \ref{cor:iter-conv} extends the results of Schmidt et al.~\cite{SLRB11} and Friedlander and Schmidt~\cite{FS12} in two ways.  First, it shows that when applied to the structured convex optimization problem~(\ref{basic-problem}), the IGM (\ref{eq:update}) can achieve an $O(1/k^2)$ convergence rate for the sequence $\{f(x^k)-f_{\rm min}\}_{k\ge0}$ even when the error norms $\{\|e^k\|_2\}_{k\ge1}$ decrease at an $O(1/k)$ rate.  This should be contrasted with the case of a general convex optimization problem, for which the IGM (\ref{eq:update}) is only known to achieve an $O(\log^2k/k)$ convergence rate for the sequence $\{\min_{0\le j\le k} f(x^j) - f_{\rm min}\}_{k\ge0}$~\cite[Proposition~1]{SLRB11}.  Secondly, our analysis shows that even when the objective function $f$ is not strongly convex, it is possible to establish a sublinear (resp.~linear) convergence rate for the sequence of iterates $\{x^k\}_{k\ge0}$, provided that the error norms $\{\|e^k\|_2\}_{k\ge1}$ decrease to zero at a sublinear (resp.~linear) rate.

\medskip
\noindent{\bf Remarks.} Since the bounds in Corollary \ref{cor:fn-root} and Theorem \ref{thm:conv-rate-gen} hold for every realization of the error sequence $\{e^k\}_{k\ge1}$, they also hold in expectation.  Thus, we can derive bounds on the expected convergence rates of $\{f(x^k)\}_{k\ge0}$ and $\{x^k\}_{k\ge0}$ whenever bounds on $\{ \E{\|e^k\|_2^2} \}_{k\ge1}$ are available.  As an illustration, we have the following extension of Corollary \ref{cor:iter-conv}:
\begin{coro} \label{cor:iter-conv-stoch} Consider the setting of Proposition \ref{prop:cost-to-go}.
\begin{enumerate}
\item[\subpb] (Expected Sublinear Convergence) Suppose that for some $\rho>0$, we have $\E{\|e^k\|_2^2} \le B_k = O\left( 1/k^{1+\rho} \right)$ for all $k\ge1$.  Then, the sequence of iterates $\{x^k\}_{k\ge0}$ satisfies
$$ \E{ f(x^{k+1}) - f_{\rm min} } \le O\left( \frac{1}{(k+1)^{1+\rho}} \right) $$
and
$$ \E{ \|x^k - x^{k+1}\|_2 } \le O\left( \frac{1}{(k+1)^{(1+\rho)/2}} \right), \quad \E{ \mbox{dist}(x^k,\mathcal{X}) } \le O\left( \frac{1}{(k+1)^{(1+\rho)/2}} \right) $$
for all $k\ge0$.

\item[\subpb] (Expected Linear Convergence) Suppose that for some $\rho \in (0,1)$, we have $\E{ \|e^k\|_2^2 } \le B_k = O(\rho^k)$ for all $k\ge1$.  Then, there exists a $c \in (0,1)$ such that the sequence of iterates $\{x^k\}_{k\ge0}$ satisfies
$$ \E{ f(x^{k+1}) - f_{\rm min} } \le O( c^{2(k+1)} ) $$
and
$$ \E{ \|x^k - x^{k+1}\|_2 } \le O(c^k), \quad \E{ \mbox{dist}(x^k,\mathcal{X}) } \le O(c^k) $$
for all $k\ge0$.
\end{enumerate}
\end{coro}
\resetspb
The proof of Corollary \ref{cor:iter-conv-stoch} can be found in Appendix \ref{app:iter-conv-stoch}.

\section{Applications to Data Fitting Problems}
Let us now apply the results in the previous section to analyze an incremental gradient method for solving least squares and logistic regression problems.  From the update formula~(\ref{eq:IncGM-update}), we see that the following approximation of $\nabla f(x^k)$ is used in the $k$--th iteration:
$$ G_k = \frac{1}{|I_k|} \sum_{i \in I_k} \nabla f_i(x^k). $$
Here, $I_k \subset \mathscr{M} \equiv \{1,2,\ldots,M\}$ is an index set that is chosen according to some pre--specified rule and corresponds to a subset of the samples.  Since both the least squares~(\ref{eq:lsq}) and logistic regression~(\ref{eq:logistic}) problems are of the form~(\ref{eq:finite-sum}), the error vector $e^{k+1}$ in the $k$--th iteration is given by
\begin{equation} \label{eq:incre-err}
e^{k+1} = G_k - \nabla f(x^k) = \frac{M - |I_k|}{M|I_k|} \sum_{i \in I_k} \nabla f_i(x^k) - \frac{1}{M} \sum_{i \in \mathscr{M} \backslash I_k} \nabla f_i(x^k).
\end{equation}
If we form $I_k$ by sampling a fixed number of elements from $\mathscr{M}$ uniformly without replacement and the sampling is done independent of $I_0,I_1,\ldots,I_{k-1}$ for all $k\ge0$, then we also have 
\begin{equation} \label{eq:incre-err-stoch}
\E{ \|e^{k+1}\|_2^2 \,\big|\, \mathscr{F}_k } = \left( \frac{M - |I_k|}{M|I_k|} \right) \left( \frac{1}{M-1} \sum_{i=1}^M \|\nabla f_i(x^k) - \nabla f(x^k)\|_2^2 \right)
\end{equation}
for all $k\ge0$, where $\mathscr{F}_k$ is the $\sigma$--algebra generated by $e^1,e^2,\ldots,e^k$ with $\mathscr{F}_0 = \emptyset$; cf.~\cite[Section 3.2]{FS12}.


\subsection{Least Squares Regression}
Recall that for the least squares regression problem~(\ref{eq:lsq}), we have $f_i(x) = (a_i^Tx-b_i)^2$ for $i=1,2,\ldots,M$.  Moreover, both Assumptions \ref{ass:general} and \ref{ass:non-empty} are satisfied and scenario (S1) holds.  Thus, in order to apply the convergence rate results in Section~\ref{sec:conv-anal}, it remains to bound the error norms $\{\|e^k\|_2\}_{k\ge1}$.  Assuming that the samples $\{a_i\}_{i=1}^M$ are uniformly bounded---i.e., there exists an $R>0$ such that $\max_{1\le i\le M} \|a_i\|_2 \le R$---we use (\ref{eq:incre-err}) to compute
\begin{eqnarray*}
 \|e^{k+1}\|_2^2 &\le& \left( \frac{M - |I_k|}{M|I_k|} \sum_{i \in I_k} \| \nabla f_i(x^k) \|_2 + \frac{1}{M} \sum_{i \in \mathscr{M} \backslash I_k} \| \nabla f_i(x^k) \|_2 \right)^2 \\
\noalign{\medskip}
&\le& \left( \frac{M - |I_k|}{M} \right)^2 \left[ \sqrt{ \frac{1}{|I_k|} \sum_{i\in I_k} \| \nabla f_i(x^k) \|_2^2 } + \sqrt{ \frac{1}{M - |I_k|} \sum_{i\in \mathscr{M} \backslash I_k} \| \nabla f_i(x^k) \|_2^2 } \right]^2 \\
\noalign{\medskip}
&\le& 8 \left( \frac{M - |I_k|}{M} \right)^2 \left( \frac{1}{|I_k|} \sum_{i \in I_k} \left( a_i^Tx^k - b_i \right)^2 \cdot \|a_i\|_2^2 \right. \\
\noalign{\medskip}
& & \qquad\qquad\qquad\qquad \left. + \frac{1}{M-|I_k|} \sum_{i \in \mathscr{M} \backslash I_k} \left( a_i^Tx^k - b_i \right)^2 \cdot \|a_i\|_2^2 \right) \\
\noalign{\medskip}
&\le& 8R^2 \left( \frac{M - |I_k|}{M} \right)^2 \left[ \frac{1}{|I_k|} \sum_{i \in I_k} \left( a_i^Tx^k - b_i \right)^2 + \frac{1}{M-|I_k|} \sum_{i \in \mathscr{M} \backslash I_k} \left( a_i^Tx^k - b_i \right)^2 \right],
\end{eqnarray*}
where the second inequality follows from the concavity of $x \mapsto \sqrt{x}$ and Jensen's inequality; the third inequality follows from the fact that $(a+b)^2 \le 2(a^2+b^2)$ for all $a,b\ge0$, and
$$ \nabla f_i(x) = 2\left( a_i^Tx - b_i \right) a_i \quad\mbox{for } i=1,2,\ldots,M. $$
Now, observe that for $M/2 \le |I_k| \le M$, we have
\begin{eqnarray*}
& & \frac{1}{|I_k|} \sum_{i \in I_k} \left( a_i^Tx^k - b_i \right)^2 + \frac{1}{M-|I_k|} \sum_{i \in \mathscr{M} \backslash I_k} \left( a_i^Tx^k - b_i \right)^2 \\
\noalign{\medskip}
&=& \frac{M}{M-|I_k|}f(x^k) + \left( \frac{1}{|I_k|} - \frac{1}{M-|I_k|} \right)  \sum_{i \in I_k} \left( a_i^Tx^k - b_i \right)^2 \\
\noalign{\medskip}
&\le& \frac{M}{M-|I_k|}f(x^k).
\end{eqnarray*}
It follows that for all $k\ge0$ with $M/2 \le |I_k| \le M$, 
$$ \|e^{k+1}\|_2^2 \le 8R^2 \frac{M-|I_k|}{M} f(x^k) = 8R^2\frac{M-|I_k|}{M} \left[ \left( f(x^k) - f_{\rm min} \right) + f_{\rm min} \right]. $$ 
Since the premises of Proposition \ref{prop:cost-to-go} are satisfied, the above inequality and (\ref{ieq:decrease}) together imply that for all $k\ge0$ with $M/2 \le |I_k| \le M$, 
\begin{eqnarray}
f(x^{k+1}) - f_{\rm min} &\le& \mu \left( f(x^k) - f_{\rm min} \right) + 8\delta R^2\frac{M-|I_k|}{M} \left[ \left( f(x^k) - f_{\rm min} \right) + f_{\rm min} \right] \nonumber \\
\noalign{\medskip}
&=& \left( \mu + 8\delta R^2 \frac{M-|I_k|}{M} \right)\left( f(x^k) - f_{\rm min} \right) + 8\delta R^2 f_{\rm min} \frac{M-|I_k|}{M}. \label{eq:lsq-recur}
\end{eqnarray}
Let
$$ \bar{\delta} = 8 \delta R^2f_{\rm min} \quad\mbox{and}\quad E_{k+1} = \frac{M-|I_k|}{M}, \quad \mu_{k+1} = \mu + 8\delta R^2E_{k+1} \quad\mbox{for }k=0,1,\ldots. $$
By applying (\ref{eq:lsq-recur}) recursively, we have
\begin{equation} \label{eq:lsq-rate}
f(x^{k+1}) - f_{\rm min} \le \left( \prod_{j=1}^{k+1} \mu_j \right) \left( f(x^0) - f_{\rm min} \right) + \bar{\delta} \left[ \sum_{j=1}^k \left( \prod_{i=j+1}^{k+1} \mu_i \right) E_j + E_{k+1} \right]
\end{equation}
for all $k\ge0$.  Now, suppose that the sets $\{I_k\}_{k\ge0}$ satisfy $M/2 \le |I_0| \le |I_1| \le \cdots$ and (ii) $\mu_1 \in (0,1)$.  Then, we have $\mu_1 \ge \mu_2 \ge \cdots$, and it follows from (\ref{eq:lsq-rate}) that for all $k\ge0$,
$$ f(x^{k+1}) - f_{\rm min} \le \mu_1^{k+1} \left( f(x^0) - f_{\rm min} \right) + \bar{\delta} \sum_{j=1}^{k+1} \mu_1^{k+1-j} E_j. $$
In particular, as long as $\{E_k\}_{k\ge1}$ decreases (sub)linearly to zero, we can use the arguments in the proofs of Theorem~\ref{thm:conv-rate-gen} and Corollary~\ref{cor:iter-conv} to show that $\{f(x^k)\}_{k\ge0}$ (resp.~$\{x^k\}_{k\ge0}$) converges at least (sub)linearly to $f_{\rm min}$ (resp.~an element in $\mathcal{X}$).  

\medskip
\noindent{\bf Remark.}  The assumptions $|I_0| \ge M/2$ and $\mu_1 \in (0,1)$ are only made for the sake of simplicity and can be dropped altogether.  Indeed, as long as $\{E_k\}_{k\ge1}$ decreases to zero, there will be an index $K\ge1$ such that $|I_k| \ge M/2$ and $\mu_k \in (0,1)$ for all $k \ge K$.  Hence, one can still derive the desired convergence rate results using (\ref{eq:lsq-rate}).

Now, suppose that the sets $\{I_k\}_{k\ge0}$ are obtained via uniform sampling from $\mathscr{M}$ without replacement.  Then, by (\ref{eq:incre-err-stoch}) and the assumption that $\max_{1\le i\le M} \|a_i\|_2 \le R$ for some $R>0$, we have
\begin{eqnarray*}
\E{ \|e^{k+1}\|_2^2 \,\big|\, \mathscr{F}_k } &\le& \frac{M - |I_k|}{M|I_k|} \left[ \frac{1}{M-1} \sum_{i=1}^M \left( \|\nabla f_i(x^k)\|_2 + \|\nabla f(x^k)\|_2 \right)^2 \right] \\
\noalign{\medskip}
&\le& \frac{M - |I_k|}{M|I_k|} \left[ \frac{1}{M-1} \sum_{i=1}^M \left( 2R\left| a_i^Tx^k-b_i \right| + \frac{1}{M}  \sum_{j=1}^M \left\| \nabla f_j(x^k) \right\|_2 \right)^2 \right] \\
\noalign{\medskip}
&\le& \frac{M - |I_k|}{M|I_k|} \left[ \frac{1}{M-1} \sum_{i=1}^M \left( 2R\left| a_i^Tx^k-b_i \right| + \sqrt{ \frac{1}{M} \sum_{j=1}^M \left\| \nabla f_j(x^k) \right\|_2^2 } \right)^2 \right] \\
\noalign{\medskip}
&\le& \frac{M - |I_k|}{M|I_k|} \left[ \frac{8R^2}{M-1} \sum_{i=1}^M \left( \left( a_i^Tx^k-b_i \right)^2 + \frac{1}{M} \sum_{j=1}^M \left( a_j^Tx^k-b_j \right)^2 \right) \right] \\
\noalign{\medskip}
&=& 16R^2 \frac{M - |I_k|}{(M-1)|I_k|} \left[ \left( f(x^k) - f_{\rm min} \right) + f_{\rm min} \right].
\end{eqnarray*}
It follows from the tower property of conditional expectation that for all $k\ge0$,
$$ \E{ \|e^{k+1}\|_2^2 } \le  16R^2 \tilde{E}_{k+1} \left( \E{ f(x^k) - f_{\rm min} } + f_{\rm min} \right), $$
where $\tilde{E}_{k+1} = (M-|I_k|)/((M-1)|I_k|)$.  If in addition the sets $\{I_k\}_{k\ge0}$ satisfy $|I_0| \le |I_1| \le \cdots$ and $\tilde{\mu}_1 = \mu + 16\delta R^2 \tilde{E}_1 \in (0,1)$, then by using a similar argument as above and Corollary~\ref{cor:iter-conv-stoch}, we see that the expected rate at which $\{f(x^k)\}_{k\ge0}$ (resp.~$\{x^k\}_{k\ge0}$) converges to $f_{\rm min}$ (resp.~an element in $\mathcal{X}$) is (sub)linear, provided that $\{\tilde{E}_k\}_{k\ge1}$ decreases (sub)linearly to zero.

\subsection{Logistic Regression}
Let us now consider the logistic regression problem~(\ref{eq:logistic}), for which we have $f_i(x) = \log(1+\exp(-b_ia_i^Tx))$ for $i=1,2,\ldots,M$.  As mentioned earlier, both Assumptions \ref{ass:general} and \ref{ass:non-empty} are satisfied, and scenario (S2) holds.  To bound the error norms $\{\|e^k\|_2\}_{k\ge1}$, we first compute
$$ \nabla f_i(x) = \frac{-b_i\exp(-b_ia_i^Tx)}{1+\exp(-b_ia_i^Tx)}a_i \quad\mbox{for } i=1,2,\ldots,M. $$
Now, assuming that $\max_{1\le i\le M} \left\{ \max\{ \|a_i\|_2,|b_i| \}\right\} \le R$ for some $R>0$, we have
$$ 
 \|e^{k+1}\|_2^2 \le \left( \frac{M - |I_k|}{M|I_k|} \sum_{i \in I_k} \| \nabla f_i(x^k) \|_2 + \frac{1}{M} \sum_{i \in \mathscr{M} \backslash I_k} \| \nabla f_i(x^k) \|_2 \right)^2 \le 4R^4E_{k+1}^2,
$$
where, as before, $E_{k+1} = (M-|I_k|)/M$.  Thus, if $\{E_k^2\}_{k\ge1}$ decreases (sub)linearly to zero, then we can directly apply Corollary~\ref{cor:iter-conv} and conclude that $\{f(x^k)\}_{k\ge0}$ (resp.~$\{x^k\}_{k\ge0}$) converges at least (sub)linearly to $f_{\rm min}$ (resp.~an element in $\mathcal{X}$).

On the other hand, if the sets $\{I_k\}_{k\ge0}$ are obtained via uniform sampling from $\mathscr{M}$ without replacement, then by (\ref{eq:incre-err-stoch}) and the assumption that $\max_{1\le i\le M} \left\{ \max\{ \|a_i\|_2,|b_i| \}\right\} \le R$ for some $R>0$, we have
\begin{eqnarray*}
\E{ \|e^{k+1}\|_2^2 \,\big|\, \mathscr{F}_k } &\le& \left( \frac{M - |I_k|}{M|I_k|} \right) \left[ \frac{1}{M-1} \sum_{i=1}^M \left( \|\nabla f_i(x^k)\|_2 + \|\nabla f(x^k)\|_2 \right)^2 \right] \\
\noalign{\medskip}
&\le& 4R^4 \frac{M - |I_k|}{(M-1)|I_k|}.
\end{eqnarray*}
This implies that for all $k\ge0$,
$$ \E{ \|e^{k+1}\|_2^2 } \le 4R^4 \tilde{E}_{k+1}, $$
where, as before, $\tilde{E}_{k+1} = (M-|I_k|)/((M-1)|I_k|)$.  Hence, using Corollary~\ref{cor:iter-conv-stoch}, we see that the expected rate at which $\{f(x^k)\}_{k\ge0}$ (resp.~$\{x^k\}_{k\ge0}$) converges to $f_{\rm min}$ (resp.~an element in $\mathcal{X}$) is (sub)linear, provided that $\{\tilde{E}_k\}_{k\ge1}$ decreases (sub)linearly to zero.

\section{Concluding Remarks}
In this paper, we considered a class of structured unconstrained convex optimization problems, in which the objective function is the composition of an affine mapping with a strictly convex function that has certain smoothness and curvature properties.  This encapsulates many problems in machine learning and data fitting, such as least squares and logistic regression.  We showed that an inexact gradient method for solving the aforementioned class of problems will converge (sub)linearly if the norms of the gradient approximation errors decrease (sub)linearly to zero.  Consequently, we were able to establish the non--asymptotic linear convergence of a growing sample--size strategy proposed in~\cite{FS12} (see also~\cite{BCNW12}) for solving the least squares and logistic regression problems.  To obtain our results, we developed a so--called global error bound, which, roughly speaking, measures the distance between a point and the optimal set in terms of some easily computable quantities.  In general, error bounds are very useful for proving strong convergence rate results for a host of optimization algorithms (see, e.g.,~\cite{LT93}).  Thus, it would be interesting to see whether such an approach can be used to exploit the structure of optimization problems arising in machine learning and establish the linear convergence of some other first--order methods.

\newpage
\section*{Appendix}
\appendix

\section{Proof of Proposition \ref{prop:suff-decrease}} \label{app:pf-suff-decrease}
Since $\nabla f$ is $L_f$--Lipschitz continuous, we have
$$
 f(x^{k+1}) - f(x^k) \le \nabla f(x^k)^T(x^{k+1}-x^k) + \frac{L_f}{2}\|x^{k+1}-x^k\|_2^2;
$$
see, e.g.,~\cite{LP66}.  Using (\ref{eq:update}) and the fact that $\alpha_k=1/L_f$ for all $k\ge0$, we obtain
\begin{eqnarray*}
f(x^{k+1}) - f(x^k) &\leq&  \frac{L_f}{2}\|x^{k+1}-x^k\|_2^2 + \left(\frac{1}{\alpha_k}(x^k-x^{k+1})-e^{k+1}\right)^T(x^{k+1}-x^k) \\
\noalign{\medskip}
&\le& -\frac{L_f}{2} \|x^{k+1}-x^k\|_2^2 + \|e^{k+1}\|_2 \|x^{k+1}-x^k\|_2,
\end{eqnarray*}
as desired.

\section{Proof of Proposition \ref{prop:pre-eb}} \label{app:pre-eb}
We begin with the following result, which is known as the Hoffman error bound:
\begin{fact}\label{prop:hoff-bd}
(cf.~\cite{H52}) Let $C \in \R^{m\times n}$ and $d \in \R^m$ be given.  Suppose that the linear system
\begin{equation}\label{eq:linear-system}
Cu=d 
\end{equation}
in $u \in \R^n$ is feasible.  Then, there exists a $\theta>0$, which depends only on $C$, such that for any $x\in\R^n$, there exists an $\bar{x} \in \R^n$ satisfying (\ref{eq:linear-system}) and 
$$ \|x-\bar{x}\|_2 \le \theta\|Cx-d\|_2. $$
\end{fact}
To prove Proposition \ref{prop:pre-eb}, consider the following linear system in $(u,v) \in \R^n \times \R^n$:
\begin{equation} \label{eq:opt-cond}
\begin{array}{rcl}
   v &=& u - E^T \nabla g(t^*), \\
   \noalign{\smallskip}
   Eu &=& t^*, \\
   \noalign{\smallskip}
   u &=& v.
\end{array}
\end{equation}
Note that $(\bar{x},\bar{x}) \in \R^n \times \R^n$ is feasible for (\ref{eq:opt-cond}) if and only if $\bar{x} \in \mathcal{X}$.  Thus, it follows from Assumption \ref{ass:non-empty} that (\ref{eq:opt-cond}) is feasible.  Now, let $z=x-\nabla f(x)=x-E^T\nabla g(Ex)$.  By Fact \ref{prop:hoff-bd}, there exist a constant $\theta>0$ and a feasible solution $(x^*,z^*)$ to (\ref{eq:opt-cond}) such that
$$ \|(x,z) - (x^*,z^*)\|_2 \le \theta\left[ \|Ex-t^*\|_2 + \|\nabla f(x)\|_2 + \left\| E^T\nabla g(Ex) - E^T\nabla g(t^*) \right\|_2 \right]. $$
Since $\| E^T\nabla g(Ex) - E^T\nabla g(t^*) \|_2 \le L \cdot \|E\| \cdot \|Ex - t^*\|_2$, the desired result follows by setting $\omega=\theta\cdot\max\{1, 1+L\|E\|\}$.

\section{Proof of Corollary \ref{cor:fn-root}} \label{app:fn-root}
Applying the inequality (\ref{ieq:decrease}) recursively yields
$$
 f(x^k)-f_{\rm min}\leq \mu^k \left( f(x^0)-f_{\rm min} \right) + \delta\sum_{j=1}^k \mu^{k-j} \|e^j\|_2^2
$$
for all $k\ge0$.  This implies that
\begin{eqnarray*}
\left| f(x^{k+1}) - f(x^k) \right| &\le& \left( f(x^k) - f_{\rm min} \right) + \left( f(x^{k+1}) - f_{\rm min} \right) \\
\noalign{\medskip}
&\le& \mu^k \left( f(x^0)-f_{\rm min} \right) + \delta \sum_{j=1}^k \mu^{k-j} \|e^j\|_2^2 \\
\noalign{\medskip}
&\quad+& \mu^{k+1} \left( f(x^0)-f_{\rm min} \right) + \delta \sum_{j=1}^{k+1} \mu^{k+1-j} \|e^j\|_2^2 \\
\noalign{\medskip}
&\le& \mu^k(1+\mu) \left( f(x^0)-f_{\rm min} \right) + \delta \left( 1+\mu^{-1} \right) \sum_{j=1}^{k+1} \mu^{k+1-j} \|e^j\|_2^2 \\
\noalign{\medskip}
&\le& \beta\sum_{j=1}^{k+1}\left( \mu^{k+1-j} \|e^j\|_2^2 + \mu^k \right),
\end{eqnarray*}
where $\beta=\max\left\{ \left( 1+\mu \right) \left( f(x^0)-f_{\rm min} \right), \delta \left( 1+\mu^{-1} \right) \right\}$.  This completes the proof.

\section{Proof of Corollary \ref{cor:iter-conv}} \label{app:iter-conv}
\begin{enumerate}
\item[\subpb] By the assumption on $\{B_k\}_{k\ge1}$, we have
\begin{equation} \label{eq:Sk}
\sum_{j=1}^k \mu^{k-j}\|e^j\|_2^2 \le \sum_{j=1}^k \mu^{k-j} O\left( \frac{1}{j^{1+\rho}} \right) 
\end{equation}
for all $k\ge1$.  To bound the quantity on the right--hand side, let us define
$$ S_k = \sum_{j=1}^k \frac{\mu^{k-j}}{j^{1+\rho}} \quad\mbox{for }k=1,2,\ldots. $$
Let $K \equiv K(\mu,\rho)>0$ be such that $\mu' = \mu(1+1/k)^{1+\rho}<1$ for all $k \ge K$, and let $C  \equiv C(\mu,\rho) \ge (1-\mu')^{-1}$ be such that $S_k \le Ck^{-(1+\rho)}$ for $k=1,2,\ldots,K$.  We now show by induction that 
\begin{equation} \label{eq:Sk-bd}
S_k \le Ck^{-(1+\rho)} \quad\mbox{for all } k\ge1. 
\end{equation}
The statement is trivially true for $k=1,2,\ldots,K$.  For $k>K$, the inductive hypothesis and our choice of $C$ imply that
$$ S_{k+1} = \mu S_k + \frac{1}{(k+1)^{1+\rho}} \le \left[ 1+C\mu\left(1+\frac{1}{k}\right)^{1+\rho} \right]\frac{1}{(k+1)^{1+\rho}} \le \frac{C}{(k+1)^{1+\rho}}. $$
This completes the inductive step.  

Now, using (\ref{eq:Sk}), (\ref{eq:Sk-bd}), Corollary \ref{cor:fn-root}, and Theorem \ref{thm:conv-rate-gen}, we have
\begin{eqnarray*}
f(x^{k+1}) - f_{\rm min} &\le& O(\mu^{k+1}) + O(S_{k+1}) \le O\left( \frac{1}{(k+1)^{1+\rho}} \right), \\
\noalign{\medskip}
\|x^k - x^{k+1}\|_2^2 &\le& O\left( S_{k+1} + \left( \frac{1+\mu}{2} \right)^k \right) = O\left( \frac{1}{(k+1)^{1+\rho}} \right), \\
\noalign{\medskip}
\mbox{dist}(x^k,\mathcal{X})^2 &\le& O\left( S_{k+1} + \left( \frac{1+\mu}{2} \right)^k \right) = O\left( \frac{1}{(k+1)^{1+\rho}} \right) 
\end{eqnarray*}
for all $k\ge0$.  This completes the proof of (a).
	
\item[\subpb] The assumption on $\{B_k\}_{k\ge1}$  implies that
$$
\sum_{j=1}^k \mu^{k-j} \|e^j\|_2^2 \le \sum_{j=1}^k \mu^{k-j} O(\rho^j) \le O(kc_1^{k}) \le O(c_2^k)
$$
for all $k\ge1$, where $c_1 = \max\{\mu,\rho\} \in (0,1)$ and $c_2 = (1+c_1)/2 \in (c_1,1)$.  Hence, by Corollary \ref{cor:fn-root} and Theorem \ref{thm:conv-rate-gen}, we have
\begin{eqnarray*}
f(x^{k+1}) - f_{\rm min} &\le& O(\mu^{k+1}) + O(c_2^{k+1}) = O(c_2^{k+1}), \\
\noalign{\medskip}
\|x^k-x^{k+1}\|_2^2 &\le& O\left( c_2^{k+1} + \left( \frac{1+\mu}{2} \right)^k \right) \le O(c_2^k), \\
\noalign{\medskip}
\mbox{dist}(x^k,\mathcal{X})^2 &\le& O\left( c_2^{k+1} + \left( \frac{1+\mu}{2} \right)^k \right) \le O(c_2^k) 
\end{eqnarray*}
for all $k\ge0$.  The desired result then follows by setting $c=\sqrt{c_2} \in (0,1)$.
\end{enumerate}
\resetspb

\section{Proof of Corollary \ref{cor:iter-conv-stoch}} \label{app:iter-conv-stoch}
By Corollary \ref{cor:fn-root} and Theorem \ref{thm:conv-rate-gen}, we have
\begin{eqnarray*}
\E{ f(x^k)-f_{\rm min} } &\le& \mu^k \left( f(x^0)-f_{\rm min} \right) + \delta\sum_{j=1}^k \mu^{k-j} \E{ \|e^j\|_2^2 }, \\
\noalign{\medskip}
\E{ \|x^k - x^{k+1}\|_2^2 } &\le& \lambda_1^2 \left[ \sum_{j=1}^{k+1} \mu^{k+1-j} \E{ \|e^j\|_2^2 } + \left( \frac{1+\mu}{2} \right)^k \right], \\
\noalign{\medskip}
\E{ \mbox{dist}(x^k,\mathcal{X})^2 } &\le& 2\kappa(1+\lambda_1^2) \left[ \sum_{j=1}^{k+1} \mu^{k+1-j} \E{ \|e^j\|_2^2 } + \left( \frac{1+\mu}{2} \right)^k \right]
\end{eqnarray*}
for all $k\ge0$.  Upon noting
$$ 
\E{ \|x^k - x^{k+1}\|_2 } \le \left( \E{ \|x^k - x^{k+1}\|_2^2 } \right)^{1/2}, \quad \E{ \mbox{dist}(x^k,\mathcal{X}) } \le \left( \E{ \mbox{dist}(x^k,\mathcal{X})^2 } \right)^{1/2}
$$
and using the assumption that $\E{\|e^k\|_2^2} \le B_k$, the rest of proof is essentially the same as that of Corollary \ref{cor:iter-conv}.

\bibliographystyle{abbrv}
\bibliography{sdpbib}

\end{document}